\documentclass[12pt,leqno,fleqn]{amsart}  
\usepackage{amsmath,amstext,amsthm,amssymb,amsxtra}
\usepackage{txfonts} 
\usepackage[T1]{fontenc}
\usepackage{lmodern}

 \usepackage{euler}   

\usepackage{tikz}

\usepackage{mathtools}
\mathtoolsset{showonlyrefs,showmanualtags}

\usepackage{hyperref} 
\hypersetup{
    colorlinks=true,       
    linkcolor=blue,          
    citecolor=magenta,        
    filecolor=magenta,      
    urlcolor=cyan           
}

\usepackage[]{amsrefs}

\allowdisplaybreaks
\swapnumbers

\theoremstyle{plain} 
\newtheorem{lemma}[equation]{Lemma} 
\newtheorem{proposition}[equation]{Proposition} 
\newtheorem{theorem}[equation]{Theorem}

\theoremstyle{definition}
\newtheorem{definition}[equation]{Definition} 

\theoremstyle{remark}
\newtheorem{remark}[equation]{Remark}

\numberwithin{equation}{section}

\title[Local $ Tb$ \& Martingale Transforms]{The Perfect Local $ Tb$ Theorem and Twisted Martingale Transforms} 
\author[MT Lacey]{Michael T. Lacey}   

\address[M.T.L]{School of Mathematics, Georgia Institute of Technology, Atlanta GA 30332, USA}
\email {lacey@math.gatech.edu}
\thanks{Research supported in part by grant NSF-DMS 0968499,  a grant from the Simons Foundation (\#229596 to Michael Lacey).
} 

\author[AV V\"ah\"akangas]{Antti V. V\"ah\"akangas}
\address[A.V.V]{School of Mathematics, Georgia Institute of Technology, Atlanta GA 30332, USA} \email{antti.vahakangas@helsinki.fi}
  \thanks{A.V.V.  was
supported by the School of Mathematics, Georgia Institute of Technology, and 
by the Finnish Academy of Science and Letters, Vilho, Yrj\"o and
Kalle V\"ais\"al\"a Foundation.
}

\begin{document}

	\maketitle

\begin{abstract} 
A local $ Tb$ Theorem provides a flexible framework for proving the boundedness of a Calder\'on-Zygmund operator $ T$. 
One needs only boundedness of the operator $ T$ on systems  of locally pseudo-accretive functions $ \{b_Q\}$, indexed by cubes.  
We give a new proof of this  Theorem in the setting of  perfect (dyadic) models of Calder\'on-Zygmund operators, imposing 
integrability conditions on the $ b_Q$ functions that  are the weakest 
possible.  The proof is a simple direct argument, based upon an   inequality for  transforms of so-called  twisted 
martingale differences,   which has been noted by Auscher-Routin. 
\end{abstract}

\section{Introduction} 

An operator $ T$ is said to be a  \emph{perfect Calder\'on-Zygmund operator} 
if it satisfies these conditions.  There is a kernel $ K (x,y)$ so that 
\begin{equation*}
	\langle T f, g \rangle= \int_{\mathbb R^{n}} \int_{\mathbb R^{n}} K (x,y) f (y) g (x) \; dy dx 
\end{equation*}
for all $ f,g$ that are smooth, compactly supported, and the closure of the supports of $ f$ and $ g$ do not intersect.  
The kernel $ K (x,y)$ is assumed to satisfy the size condition 
\begin{equation} \label{e.size}
	\lvert 	K (x,y)\rvert \le \frac 1 {\lvert  x-y\rvert ^{n} } 
\end{equation}
and it satisfies the following strong smoothness condition.  For any two disjoint dyadic cubes $ P,Q$, $ K (x,y)$ is constant on $ P \times Q$. 
The implication of this property, used repeatedly, is this: If $ f$ is supported on $ P$ and $ g$ is supported on $ Q$, and at least 
one of $ f$ and $ g$ have integral zero, then $	\langle T f, g \rangle=0 $.

We are interested in bounded Calder\'on-Zygmund operators, so we set $ \mathbf T$ to be the  norm of $ T$ on $ L^2 (\mathbb R ^{n})$, 
namely $ \mathbf T$ is the best constant in the inequality 
\begin{equation*}
	\bigl\lvert 	\langle T f, g \rangle\bigr\rvert \le \mathbf T \lVert f\rVert_{2} \lVert g\rVert_{2} \,. 
\end{equation*}
It is well known that this inequality  extends to
the form $\bigl\lvert 	\langle T f, g \rangle\bigr\rvert \lesssim \mathbf T \lVert f\rVert_{p} \lVert g\rVert_{p'}$, 
where $1<p<\infty$ and $1/p+1/p'=1$.

The celebrated $ T1$ Theorem of David and Journ\'e \cite{MR763911} gives a beautiful characterization of the bounded 
Calder\'on-Zygmund operators.  It was the powerful observation of Michael Christ \cite{MR1096400} that a 
weakening of the $ T1$ Theorem, to a so-called $ Tb$ formulation, 
can have wide-ranging implications.  Christ himself addressed analytic capacity, and it 
was this perspective that was crucial to the solution of the Kato square root problem \cites{MR1933725,MR1933726}.  
Our focus is on the local $ Tb$, in the dyadic model, as promoted in \cite{MR1934198}. 
  This is the usual definition of systems  of accretive functions.

\begin{definition}\label{d.system}
	Fix $ 1<p < \infty $.   
	A collection of functions $ \{b _Q \;:\; Q\in \mathcal D\} $ is  called \emph{a system of $p$-accretive functions with constant $ 1< \mathbf A$} if these conditions hold for each dyadic cube $ Q\in\mathcal{D}$. 
\begin{enumerate}
\item  $ b_Q$ is supported ond $ Q$ and $ \int _{Q} b_Q (x) \; dx = \lvert  Q\rvert $. 
\item  $ \lVert b_Q\rVert_{p} \le \mathbf A|Q|^{1/p}$. 
\end{enumerate}
\end{definition}

In the Theorem below  $ p_1, p_2$ are not related by duality;  for instance it is allowed that $ 1< p_1,p_2 < 2$. 

\begin{theorem}\label{t.main}  
For  fixed constants 
	$ \mathbf A$ and $\mathbf{T}_{\textup{loc}}$, this holds.  
	Suppose that $ T$ is a perfect dyadic Calder\'on-Zygmund operator and,
	for $ 1< p_1, p_2< \infty $, there
	are systems $ \{b ^{j} _{Q}\}$ of $p_j$-accretive functions with constant $ \mathbf A$, 
	so that 
	\begin{equation*}
		\int _{Q} \lvert  T b_Q^1\rvert ^{p_2'} \; dx  \le \mathbf T ^{p_2'} _{\textup{loc}} \lvert  Q\rvert \,, 
		\qquad 
		\int _{Q} \lvert  T^* b_Q^2\rvert ^{p_1'} \; dx  \le \mathbf T ^{p_1'} _{\textup{loc}}\lvert  Q\rvert.
\end{equation*}
Then, $ T$ extends to a bounded operator on $ L ^2 $, and moreover, 
$ \mathbf T \lesssim_{\mathbf{A},p_1,p_2} 1+   \mathbf T _{\textup{loc}}$.
\end{theorem}

This is a known result, \cite[Theorem 6.8]{MR1934198}. 
Auscher and Routin  \cite[Section 8]{1011.1747} have recently devised a  proof  closely related to this one.  

Martingale transform inequality for twisted differences play the central role. 
These inequalities have also been used by Auscher-Routin \cite{1011.1747}*{Section 5}. 
The direct proof of the  the Theorem  proceeds by standard reductions, and  construction of stopping cubes from the local $ Tb$ hypotheses, and  a brief additional argument. 
A highlight is a simple appeal to the local $ Tb$ hypothesis and the martingale transform inequality. 
Compare to \cite{1011.1747}*{Estimate for $ \langle f, V _{1,2}g \rangle$}.  

Relevant history, and indications of the utility of $ Tb$ theorems can be found in surveys by S. Hofmann \cites{MR2275650,MR2664559}. 
See in particular  \cite{MR2664559}*{\S3.3.1}, where the extension of the Theorem above to the continuous case is specifically mentioned. 
The perfect case is of course very special, still the argument in \cite{MR1934198} has been influential, although 
the task of lifting the proof therein to the continuous case has not proven to be easy.  
Auscher and  Yang \cite{MR2474120} succeeded in extending the Theorem above to the continuous case, 
with the duality assumption on $ p_1$ and $ p_2$  but the argument 
is an indirect reduction to the perfect case.  This is less desirable, due to the interest in local $ Tb$ theorems 
more general settings, such as the setting of homogeneous spaces, as in  Auscher and Routin \cite{1011.1747}.  
The latter paper employs the  Belykin-Coifman-Rohklin algorithm, see \cite{MR1085827,MR1110189}.
The latter paper addresses the  the case where $ 1/p_1 + 1/p_2 > 1$,  but  additional hypotheses are needed, and their nature is still unresolved.    
 One can also consult  Hyt\"onen-Martikainen   \cites{0911.4387,1011.0642} for another general approach to local $ Tb$ Theorem in non-homogeneous and upper doubling settings, although in the setting where duality is imposed.  
A local $Tb$ Theorem  in a vector-valued setting,
with strong conditions on accretive functions,
is considered in \cite{1201.0648}. 
Salamone \cite{MR2666228} also studies the dyadic $ Tb$ Theorem.  

Notation: For any cube $ Q$, $ \langle f  \rangle_Q := \lvert  Q\rvert ^{-1} \int _{Q} f \; dx  $, 
and $ \ell Q = \lvert  Q\rvert ^{1/n} $ is the side length of the cube. $ A \lesssim B$ means that 
$ A \le C \cdot B$, where $ C$ is an unspecified constant independent of $ A$ and $ B$.

\section{The Martingale Transform Inequality} 

The  classical martingale transform inequality is this. 
For all constants satisfying $\lvert\varepsilon_Q\rvert\le 1$,
\begin{equation}\label{e.classical_martingale}
	\Bigl\lVert  \sum_{Q\in \mathcal D} \varepsilon _Q \sum_{Q'\in\textup{ch}(Q)}
	\big\{\langle f\rangle_{Q'}-\langle f\rangle_{Q}\big\}\mathbf{1}_{Q'}
	 \Bigr\rVert_{q}  \lesssim
	 \lvert\lvert f\rvert\rvert_{q},\quad 1<q<\infty.
\end{equation}
A variant  is stated in Theorem \ref{t.mt}, and it is essential to the  subsequent arguments.  
This section can be read independently of the rest of the paper.
Above, and for the remainder of the paper, $ \langle f \rangle_Q := \lvert  Q\rvert ^{-1} \int _{Q} f \; dx $ 
is the average of $ f$ on cube $ Q$.

Fix a function $ b$ supported on a dyadic cube $ S_0$, satisfying $ \int b \; dx = \lvert  S_0\rvert $, and 
$ \lVert b\rVert_{p} \le \mathbf A\lvert  S_0\rvert ^{1/p} $, where $ 1<p< \infty $ is fixed. 
Let $ \mathcal T'$ be the maximal dyadic cubes $ T\subset S_0$ which 
meet either one of these conditions with $\delta\in (0,1)$:
\begin{equation}\label{e.ra}
	\Bigl\lvert \int _{T} b \; dx  \Bigr\rvert \le \delta   \lvert  T\rvert\, \quad \textup{or} \quad  
	\int _{T} \lvert b\rvert ^{p}\,dx \ge \delta^{-1}  \mathbf A ^{p}\lvert  T\rvert\,.  
\end{equation}
We will consider a fixed but arbitrary family $\mathcal{T}$ of
disjoint dyadic cubes in $S_0$, the `terminal cubes', and our estimates are not allowed depend upon this family.
Moreover, we require that $T'\subset T\in\mathcal{T}$ if $T'\in\mathcal{T}'$.
To each terminal cube $ T$, we have a function $ b_T$ supported on $ T$, and  satisfying 
$\int b_T\; dx=|T|$ and $ \lVert b_T\rVert_{p} \le \mathbf A|T|^{1/p}$.

Let $ \mathcal Q$ be all dyadic cubes, contained in $S_0$, but not contained in any $ T\in \mathcal T$. Define 
\begin{equation*}
	\Delta _{Q} f := 
		\sum_{Q'\in \textup{ch}(Q)} 
	\Bigl[
	\frac {\langle f \rangle _{Q'}} {\langle b_{Q'} \rangle _{Q'} } b _{Q'}
	-
	\frac {\langle f \rangle _{Q}} {\langle b \rangle_Q}  b 
	\Bigr]  \mathbf 1_{Q'} \,, \qquad Q\in \mathcal Q\,, 
\end{equation*}
where we set $ b _{Q'} = b$ if $ Q'\not\in \mathcal T$ and otherwise, 
$ b_{Q'}$ is defined as above.  We refer to these as the \emph{twisted martingale differences.}

\begin{theorem}\label{t.mt} 
	This inequality holds for all selection of constants $ \lvert  \varepsilon _Q\rvert \le1$. 
\begin{equation}\label{e.mt}
	\Bigl\lVert  \sum_{Q\in \mathcal Q} \varepsilon _Q \Delta _Q f \Bigr\rVert_{p}  \lesssim   \lVert f\rVert_{p}\,, 
\end{equation}
where $1<p<\infty$ is the exponent associated with the function $b$.
\end{theorem}

This Theorem and Theorem~\ref{t.LM} below are contained in \cite{1011.1747}*{Lemma 5.3}.  
A randomized version of this Theorem in a vector-valued context is proven
in \cite[Section 4]{1201.0648}. And, the more common square function variant is well-known. 
\color{black}
We give a somewhat different proof, in the spirit of completeness, since we view the inequality 
as fundamental to the  $ Tb$ theorems. \color{black}

We need the following preparation.  In the sum below, we do not sum over the children which are 
terminal cubes, and we do not multiply by the $ b$ functions, and so we refer to these as the \emph{half-twisted differences}. 
\begin{equation} \label{e.D}
D_{Q} f := 
\sum_{Q'\in \textup{ch}(Q)\setminus \mathcal T} 
	\Bigl[
	\frac {\langle f \rangle _{Q'}} {\langle b \rangle _{Q'}   } 
	-
	\frac {\langle f \rangle _{Q}} {\langle b \rangle_Q}  
	\Bigr]  \mathbf 1_{Q'} \,, \qquad Q\in \mathcal Q\,. 
\end{equation}
The following universal estimate holds in Lebesgue measure. 

\begin{theorem}\label{t.LM} 
	These inequalities hold for all selection of constants $ \lvert  \varepsilon _Q\rvert \le1$. 
	\begin{equation}\label{e.lm}
	\Bigl\lVert  \sum_{Q\in \mathcal Q} \varepsilon _Q  D _Q f \Bigr\rVert_{q} \lesssim   \lVert f\rVert_{q} 
	\,, \qquad 1< q < \infty \,. 
\end{equation}
\end{theorem}

\begin{proof} 
	It is important to note that this operator is, in fact, a constant multiple of a perfect Calder\'on-Zygmund operator. 
	It therefore suffices to verify the conditions of the $ T1$ Theorem, but this is not 
	convenient to do directly.  Instead, we  write the operator as a sum of three perfect Calder\'on-Zygmund 
	operators. In verifying the $ T1$  conditions for these operators,  we  use the 
	$ \mathbf T _{\textup{weak}}$ constant, testing the $L ^{1}$ and/or $L^p$ norm of $ T \mathbf 1_{F} $ and $ T ^{\ast} \mathbf 1_{F}$ 
	for cubes $F$.  Recall that $ b \in L ^{p}$. 

	For cube $ Q\in \mathcal Q$ with child $ Q'\not\in \mathcal T$, we  write 
\begin{align}
	\frac {\langle f \rangle _{Q'}} {\langle b \rangle _{Q'}  } 
	-
	\frac {\langle f \rangle _{Q}} {\langle b \rangle_Q}  
	&=   
	 \Bigl\{
	\frac {\langle f \rangle _{Q'}} {\langle b \rangle _{Q}  } 
	-	\frac {\langle f \rangle _{Q}} {\langle b \rangle_Q}  
	\Bigr\}+ 
	\Bigl\{
		\frac {\langle f \rangle _{Q'}} {\langle b \rangle _{Q'}  } - 
			\frac {\langle f \rangle _{Q'}} {\langle b \rangle _{Q}  } 
	\Bigr\} 
\\ &=   
	 \Bigl\{
	\frac {\langle f \rangle _{Q'}} {\langle b \rangle _{Q}  } 
	-	\frac {\langle f \rangle _{Q}} {\langle b \rangle_Q}  
	\Bigr\}
 +	\bigl\{ {\langle b \rangle_Q}  -  {\langle b \rangle _{Q'}  }  \bigr\} 
 \frac {\langle f \rangle_{Q'}} {  {\langle b \rangle _{Q'}  } {\langle b \rangle_Q}  } 
	\\ \label{e.mt1}
	&= 
	\Bigl\{
		\frac {\langle f \rangle _{Q'}} {\langle b \rangle _{Q}  } - 
			\frac {\langle f \rangle _{Q}} {\langle b \rangle _{Q}  } 
	\Bigr\} 
	\\ \label{e.mt2}
	& \qquad +	\bigl\{ {\langle b \rangle_Q}  -  {\langle b \rangle _{Q'}  }  \bigr\} 
	\frac {\langle f \rangle_{Q'}} {  {\langle b \rangle _{Q} ^2   }  } 
	\\ \label{e.mt3}& \qquad + 
	\bigl\{ {\langle b \rangle_Q}  -  {\langle b \rangle _{Q'}  }  \bigr\} ^2 
	\frac {\langle f \rangle_{Q'}} {  {\langle b \rangle _{Q'}  } {\langle b \rangle_Q} ^2   } 
\end{align}
This gives us three sums to bound.  Keep in mind that the averages of $ b$ that occur are bounded from 
above and below by failure of \eqref{e.ra}. In the first two expressions, the denominator is only a function of $ Q$, while 
in the third, it depends upon the child $ Q'$, with however the square on the difference on $ b$. 
The first term gives rise to a  classical martingale difference on $ f$, the second a martingale difference 
on $ b$,  and the third, 
a square function of a martingale difference on $ b$.

Let us observe that
\begin{equation}\label{e.simple_ineq}
	\Bigl \lVert 	\sum_{Q \in \mathcal Q} \varepsilon_Q \sum_{Q'\in\textup{ch}(Q)\setminus\mathcal{T}} 
	\big\{\langle f\rangle_{Q'}-\langle f\rangle_{Q}\big\}\mathbf{1}_{Q'}
	\Bigr\rVert _q 
	\lesssim ||f||_q \,, \qquad 1< q < \infty   \,. 
\end{equation}
Indeed, this is a consequence of the classical martingale transform inequality \eqref{e.classical_martingale}
and  maximal function
estimates in the disjoint family of missing terminal cubes. 
The desired estimate for the sum associated with terms \eqref{e.mt1} follows from this.

\smallskip 

An estimate for the sum associated with term \eqref{e.mt2} is clearly a consequence
of inequality,
\begin{equation}\label{e.SIO_ineq}
	\Bigl \lVert 	\sum_{Q \in \mathcal Q} \varepsilon_Q \sum_{Q'\in\textup{ch}(Q)\setminus\mathcal{T}} 
	\big\{\langle b\rangle_{Q'}-\langle b\rangle_{Q}\big\}\langle f\rangle_{Q'}\mathbf{1}_{Q'}
	\Bigr\rVert _q 
	\lesssim \mathbf{A} ||f||_q \,, \qquad 1< q < \infty   \,. 
\end{equation}
Denote the linear operator on the left hand side by $\Pi$.
After normalizing with a constant $c_{n,\delta} \mathbf{A}^{-1}$, 
we are looking
for $L^q$-norm estimates for a symmetric perfect Calder\'on-Zygmund operator.
It is classical that it suffices to verify inequality,
\begin{equation}\label{e.pi_rweak_type_app}
\lvert\lvert\Pi \mathbf 1_F \rvert\rvert_{L^1(F)} \lesssim \lvert F\rvert\,,
\end{equation}
where $F$ is a dyadic cube. In order to do this, let us write
\begin{equation}\label{e.two_series}
\Pi \mathbf{1}_F = \bigg\{\sum_{\substack{Q \in \mathcal Q\\ Q\supsetneq F}}
+\sum_{\substack{Q \in \mathcal Q\\ Q\subset F}}\bigg\}
 \varepsilon_Q \sum_{Q'\in\textup{ch}(Q)\setminus\mathcal{T}} 
	\big\{\langle b\rangle_{Q'}-\langle b\rangle_{Q}\big\}\langle \mathbf 1_F\rangle_{Q'}\mathbf{1}_{Q'}.
\end{equation}
Using Minkowski inequality and the trivial estimate 
$\lvert\langle \mathbf{1}_F\rangle_{Q'}\rvert\le  \lvert F\rvert \lvert Q'\rvert^{-1}$, we find
that $L^p$-norm of the first series is bounded 
by $c_n\delta^{-1} \mathbf{A}\lvert F\rvert^{1/p}$.
Concerning the second series, we can clearly assume that $Q\subset  F$ for some $Q\in\mathcal{Q}$.
Let us denote by
$R$ the maximal cube in $\mathcal{Q}$, contained in $F$.
Assuming $\mathcal{Q}\ni Q\subset F$ and $Q'$ is a child of $Q$, then $\langle \mathbf{1}_F\rangle_{Q'}=1$,
 $\langle b\rangle_{Q'}=\langle b\mathbf{1}_R\rangle_{Q'}$,
and likewise $\langle b\rangle_{Q}=\langle b\mathbf{1}_R\rangle_{Q}$. By inequality \eqref{e.simple_ineq}, 
setting $\varepsilon_Q=0$ if $Q\supsetneq F$,
the $L^p$-norm
of the second series in \eqref{e.two_series} is
bounded by $\lvert\lvert b\mathbf{1}_R\rvert\rvert_p$ which, in turn,
is bounded by $\delta^{-1/p}\mathbf{A}|F|^{1/p}$. 
This concludes
the proof of inequality \eqref{e.pi_rweak_type_app} and, as a consequence,
we obtain inequality \eqref{e.SIO_ineq}.

\smallskip 

It remains to estimate the sum associated with term \eqref{e.mt3}. Namely,
we need the following inequality,
\begin{equation}\label{e.SIO_ineq_2}
	\Bigl \lVert 	\sum_{Q \in \mathcal Q} \varepsilon_Q \sum_{Q'\in\textup{ch}(Q)\setminus\mathcal{T}} 
	\big\{\langle b\rangle_{Q'}-\langle b\rangle_{Q}\big\}^2
	\frac{\langle f\rangle_{Q'}}{\langle b\rangle_{Q'}\langle b\rangle_Q^2}\mathbf{1}_{Q'}
	\Bigr\rVert _q 
	\lesssim \mathbf{A} ||f||_q \,, \qquad 1< q < \infty   \,. 
\end{equation}
Denote the linear operator on the left hand side by $\amalg$.
Again, after a normalization by $c_{n,\delta} \mathbf{A}^{-2}$, we are
looking for $L^q$ estimates of perfect Calder\'on-Zygmund operator.
By symmetry of $\amalg$, it suffices to verify that
\begin{equation}\label{e.pi_rweak_type}
\lvert\lvert  \amalg \mathbf{1}_F\rvert\rvert_{L^1(F)}  \lesssim |F|\,,
\end{equation}
where $F$ is a dyadic cube. In order to verify this
inequality, we split the series defining $\amalg \mathbf{1}_F$ in two parts as above, one
with cubes $Q\supsetneq F$ and the other with cubes $Q\subset F$.
Reasoning as above,
we find that the $L^2$-norm of the first series is bounded by $c_{n,\delta}\mathbf{A}^2|F|^{1/2}$.
The second series to estimate is
\begin{align*}
	&\Bigl \lVert 	\sum_{\substack{Q \in \mathcal Q\\Q\subset F}} \varepsilon_Q \sum_{Q'\in\textup{ch}(Q)\setminus\mathcal{T}} 
	\big\{\langle b\rangle_{Q'}-\langle b\rangle_{Q}\big\}^2
	\frac{\langle \mathbf 1_F\rangle_{Q'}}{\langle b\rangle_{Q'}\langle b\rangle_Q^2}\mathbf{1}_{Q'}
	\Bigr\rVert _1\,.
\end{align*}
Using the fact that $\langle \mathbf 1_F\rangle_{Q'}=1$ if $Q\subset F$ and $Q'$ is a child of $Q$, yields the upper bound
\begin{align*}
	&c_\delta
	\sum_{\substack{Q \in \mathcal Q}} \sum_{Q'\in\textup{ch}(Q)\setminus\mathcal{T}}
	\int_{\mathbf R^n}
	\big\lvert \big\{\langle b\rangle_{Q'}-\langle b\rangle_{Q}\big\}
	{\langle \mathbf 1_F\rangle_{Q'}}\mathbf{1}_{Q'}(x)\big\rvert^2\,dx\,,
\end{align*}
where we have also relaxed the summation condition using positivity of the summands.
This upper bound, in turn, is bounded by a constant
multiple of $\mathbf{A}^2\lvert\lvert \mathbf 1_F\rvert\rvert_2^2=\mathbf{A}^2|F|$ -- 
a consequence of a square function estimate arising from randomization
of inequality
\eqref{e.SIO_ineq} with $q=2$ by taking $\varepsilon_Q$'s to be Rademacher random variables.
This completes the proof of inequality \eqref{e.pi_rweak_type} and, consequently, the proof of theorem.
\end{proof}

We have the following easier proposition.

\begin{proposition}\label{p.LM} 
This inequality holds for all selection of constants $|\varepsilon_Q|\le 1$.
	\begin{equation*}
		\Bigl\lVert \sum_{T\in \mathcal T} \varepsilon_{T^{(1)}}\langle f \rangle_T \mathbf 1_{T}
		\Bigr\rVert _q + 
		\Bigl\lVert \sum_{T\in \mathcal T} \varepsilon_{T^{(1)}}\frac{\langle f \rangle_{T^{(1)}}}{\langle b\rangle_{T^{(1)}}} \mathbf 1_{T}
		\Bigr\rVert _q
		\lesssim \lVert f\rVert_q,\quad 1<q<\infty\,.
	\end{equation*}
\end{proposition}

\begin{proof}
By disjointness
of terminal cubes and the estimate $|\langle b\rangle_{T^{(1)}}|\ge \delta$ for $T\in\mathcal{T}$,
\begin{align*}
\Bigl\lVert \sum_{T\in \mathcal T} \varepsilon_{T^{(1)}}\frac{\langle f \rangle_{T^{(1)}}}{\langle b\rangle_{T^{(1)}}} \mathbf 1_{T}
		\Bigr\rVert _q &\le \delta^{-1} \bigg(\sum_{T\in\mathcal{T}} \int_{\mathbf R^n}  \big\lvert
		\langle f\rangle_{T^{(1)}} \mathbf 1_T(x)\big\rvert^q\,dx\bigg)^{1/q}\\
		&\lesssim \delta^{-1}\lvert\lvert Mf\rvert\rvert_q\lesssim \delta^{-1}\lvert\lvert f\rvert\rvert_q.
\end{align*}
The other term is estimated in a similar manner.
\end{proof}

\begin{proof}[Proof of Theorem~\ref{t.mt}.]  
	Let us set $ B f := \sum_{Q\in \mathcal Q} \varepsilon_Q D_Q f$, and observe
	that
	\begin{equation}\label{e.delta_decomp}
	\sum_{Q\in\mathcal{Q}} \varepsilon\Delta_Q f
	=Bf \cdot b + 
	\sum_{T\in\mathcal{T}} \varepsilon_{T^{(1)}} \langle f\rangle_T \mathbf 1_T\cdot \sum_{T\in\mathcal{T}} b_T
	-\sum_{T\in\mathcal{T}} \varepsilon_{T^{(1)}} \frac{\langle f\rangle_{T^{(1)}}}{\langle b\rangle_{T^{(1)}}}
	 \mathbf 1_T\cdot b\,.	
	\end{equation}
		Consider the events $ E _{\lambda } := \{ \lvert  Bf \rvert \ge \lambda  \}\subset S_0$, 
	where $ \lambda >0$. 
	Let $S_\mathcal{T}\subset S_0$ be the union of terminal cubes $T\in \mathcal{T}$.
By construction of $\mathcal{T}$, and Lebesgue differentiation Theorem,
we have $|b(x)|\le \delta^{-1}\mathbf{A}$ for almost every $x\in S_0\setminus S_{\mathcal{T}}$.
Hence,
\begin{align*}
\int_{E_\lambda\setminus S_{\mathcal{T}}} |b|^p\,dx \le 
\delta^{-p}\mathbf{A}^p|E_\lambda \setminus S_{\mathcal{T}}|\,.
\end{align*}
Observe that $Bf$ is constant on terminal cubes. 
Let us denote $\mathcal{T}_\lambda:=\{T\in\mathcal{T}\,:\,|Bf|\ge \lambda\textrm{ on }T\}$. Since
$T^{(1)}\in \mathcal{Q}$ for each terminal cube $T$,
\begin{align*}
\int_{E_\lambda\cap S_{\mathcal{T}}} |b|^p\,dx &=\sum_{T\in\mathcal{T}_\lambda}
|T| \bigg(\frac{1}{|T|} \int_T |b|^p\,dx\bigg) \\&\le 2^n\delta^{-p}\mathbf{A}^p \sum_{T\in\mathcal{T}_\lambda}|T|
\le 2^n\delta^{-p}\mathbf{A}^p |E_\lambda\cap S_{\mathcal{T}}|.
\end{align*}
 It follows that we
	can compare Lebesgue measure estimates and estimates with respect to $ \lvert  b (x)\rvert ^{p} \; dx  $. 
	Namely, 
\[
\int_{E_\lambda} |b|^p\,dx \le 2^n\delta^{-p}\mathbf{A}^p |E_\lambda|.
\]
Therefore, by a standard formula
and the Lebesgue measure estimates of Theorem~\ref{t.LM},  
\begin{equation*}
	\int _{S_0} \lvert  Bf\rvert ^{p } \lvert  b\rvert ^{p} \; dx  
	= p\int _{ 0} ^{\infty } \lambda ^{p-1} \int _{E _{\lambda }} \lvert  b\rvert ^{p} \; dx  \; d \lambda 
	\lesssim \mathbf A^{3p} \lVert f \rVert_p ^{p} \,. 
	\end{equation*}

The $L^p$-norms of the two remaining quantities in the right
hand side of \eqref{e.delta_decomp} are estimated in a similar manner, by using
 Proposition~\ref{p.LM} and measures
$\sum_{T\in \mathcal T} \lvert  b_T\rvert ^{p} \, dx$ and 
	$\lvert  b (x)\rvert ^{p} \,dx$ instead.

\end{proof}

\section{The Corona} 

We will work with different subsets of the dyadic grid, and need some 
notations.  
Given $ Q \in \mathcal D$, we denote by $ \textup{ch}(Q)$ the $ 2 ^{n}$ dyadic children of $ Q$. 
Given $ \mathcal S\subset \mathcal D$, we can refer to the $ \mathcal S$-children of $ S\in \mathcal S$: 
The maximal elements $ S'$ of $ \mathcal S$ that are strictly contained in $ S$.  This collection is denoted as 
$ \textup{ch} _{\mathcal S} (S)$. 
For a
cube $ Q\in \mathcal D$, that is contained in a cube in $\mathcal{S}$,
we take $ \pi _{\mathcal S} Q$ to be the $ \mathcal S$-parent of $ Q$: The smallest
cube in $ \mathcal S$ that 
contains $ Q$.  

This is the construction of stopping cubes: for a fixed $Q_0\in\mathcal{D}$, families
$\mathcal{S}_1, \mathcal{S}_2\subset \mathcal D$ are defined as follows. 
Take the cube
$ Q_0$ in $\mathcal{S}_1$.
In the inductive stage, if $ S\in \mathcal S_1$, take as members of $ \mathcal S_1$ 
those maximal dyadic descendants  $ Q$ which meet any one of these several conditions: 
\begin{enumerate}  

\item  $\big\rvert \int_Q b^{1}_{S} \big\rvert\le \delta|Q|$.

\item $    \int _{Q} \lvert b_S^{1}\rvert ^{p_1} \ge \delta ^{-1}\mathbf{A}^{p_1} \lvert  Q\rvert$.

\item  $    \int _{Q} \lvert T b_S ^{1}\rvert ^{p_2'} \ge \delta ^{-1}\mathbf T _{\textup{loc}} ^{p_2'} \lvert  Q\rvert $. 
\end{enumerate}
A stopping tree $\mathcal{S}_2$
is then constructed
analogously, but using functions
$\{b^2_S\}_{S\in\mathcal{D}}$, 
and exponents $p_2$ and $p_1'$ in conditions (2) and (3),
respectively. Furthermore, in (3) we use $T^*$ instead of $T$.

If $\delta>0$ is chosen sufficiently small in the construction above,
then there is a constant $\tau\in (0,1)$
such that
\begin{equation}\label{e.generations_no_f}
	\sum_{\substack{S' \in \textup{ch} _{\mathcal S_1} (S) }} |S'| \le \tau|S|\,, \qquad S\in \mathcal S_1\,. 
\end{equation}
Here both $\tau$ and $\delta$ depend on $\mathbf{A}$. 
It follows 
from construction of $\mathcal{S}_1$ and inequality \eqref{e.generations_no_f}  --
and their counterparts in $\mathcal{S}_2$ --
that $ \mathcal S$ is a Carleson family of of cubes.  Namely, 
there holds
\begin{equation}\label{e.carleson}
	\sum_{S \in \mathcal S \;:\; S\subset Q} \lvert  S\rvert \lesssim_{\tau} \lvert  Q\rvert  \,, \qquad \mathcal{D}\ni Q\subset Q_0\,.
\end{equation}
In the sequel, we suppose that $\delta$ is chosen as above.

\begin{remark}
With $ \mathcal S$ so constructed, many subsequent inequalities have constant 
that depend upon the values of $ \delta $ and $ \mathbf A$.  The dependence is 
not straightforward, and we do not attempt to track it.  Frequently, this dependence is 
even suppressed in the notation, so that in many parts of the argument relative to $ \mathcal S$, 
the symbol `$ \lesssim  $'  should be read as `the unspecified implied constant depends upon 
dimension, $\tau$, $ \delta $ and $ \mathbf A$, but is otherwise absolute.' 
Dependencies on other parameters are indicated by subscripts,
 e.g. $\lesssim_p 1$
means $\lesssim c_p$ where  $c_p$ depends on $p$ only.
\end{remark}

Following \cite{MR986825}, we define the $ b$-adapted conditional expectations and martingale differences,
associated with a dyadic cube $Q\subset Q_0$, by
\begin{align}\label{e.Delta}
E_Q^j h := \frac {\langle h \rangle _{Q}} {\langle b ^{j} _{\pi_{\mathcal{S}_jQ}} \rangle_Q} 
	b ^{j}_{\pi_{\mathcal{S}_j}Q} \mathbf{1}_Q,\qquad
	\Delta _{Q} ^{j} h &:= 
	\sum_{Q'\in \textup{ch}(Q)} 
	\bigl[
	E_{Q'}^j h -E_{Q}^j h\bigr] \cdot \mathbf{1}_{Q'}\,.
\end{align}
Observe that the functions $\Delta_Q^j h$ have mean zero in $\mathbf{R}^n$, and
they are supported in the cube $Q$. 

The following Lemma 
is proven like  Lemma 3.5 in \cite{1011.0642}, with obvious modifications.

\begin{lemma}\label{l.representation}
Let 
$h\in L^{p_j}(\mathbf{R}^n)$, $j\in \{1,2\}$. Then, there holds pointwise and in 
$L^{p_j}(\mathbf{R}^n)$, 
	\begin{equation*}
		h\mathbf{1}_{S}= E_{S}^j h + \sum_{Q:Q\subset S}  \Delta _{Q}^j  h,\,\qquad 
		\quad S\subset Q_0.
\end{equation*}
\end{lemma}

We will suppress the notation by denoting
$E_Q^j=E_Q$ and $\Delta_Q^j = \Delta_Q$, $j\in\{1,2\}$.

\medskip

By a well-known reduction from the  $T1$ Theorem,    it suffices to show that
\[
\lvert \langle Tf, g\rangle\rvert\lesssim (1+\mathbf{T}_{\textup{loc}})|Q_0|,
\]
where $f$ and $g$ are measurable functions with the property $|f|=|g|= \mathbf{1}_{Q_0}$, \cite{MR1934198,1011.1747,MR2664559}. 
By
Lemma \ref{l.representation},
the expansion of the bilinear form is
\begin{align*}
	\langle T f,g \rangle &= \langle T E_{Q_0} f,g\rangle 
	+ \langle T\sum_{P:P\subset Q_0} \Delta_P f,E_{Q_0} g\rangle
	+ \sum_{P,Q:P\cup Q\subset Q_0} \langle T \Delta_P f,\Delta_Q g\rangle.
\end{align*}
Using the assumptions, it is straightforward to verify that
\begin{align*}
|\langle T E_{Q_0} f,g\rangle|\le  \mathbf{T}_{\textup{loc}} |Q_0|^{1/p_2'} |Q_0|^{1/p_2}
\end{align*}
and, by using also Lemma \ref{l.representation}, that
\begin{align*}
|\langle T\sum_{P:P\subset Q_0} \Delta_P f,E_{Q_0} g\rangle|
\le \mathbf{A}\mathbf{T}_{\textup{loc}}|Q_0|^{1/p_1}|Q_0|^{1/p_1'}.
\end{align*}
Since $1/p_1+1/p_1'=1=1/p_2+1/p_2'$, we are left with estimating the main term
\begin{align*}
\sum_{P, Q} \langle T \Delta _{P} f, \Delta _{Q} g \rangle=
	\Biggl\{ \sum_{ P, Q \;:\; ell P <  ell Q} + 
	\sum_{P, Q \;:\; ell P = ell Q} +
	\sum_{ P, Q \;:\; ell P> ell Q}    \Biggr\} \langle T \Delta _{P} f, \Delta _{Q} g \rangle\,,
	\end{align*}
where all the summations are restricted to dyadic cubes $P,Q$ contained in $Q_0$.
As is standard, we will assume that $ T$ has kernel $ K$ which is perfect, and $ K (x,y)$ is 
identically zero for $ \lvert  x-y\rvert $ sufficiently small. In particular, the sum above can be taken over a 
finite collection of $ P, Q$. We will rearrange the sum, as is convenient for us. All bounds are independent of the 
this last assumption on the kernel.

We will show that the second and third forms above obey the desired estimate, and by 
duality the same will hold for the first term. 
%
To state this otherwise, it suffices to consider 
the two forms 
\begin{align}
	B _{\textup{above}} (f,g) & := 
	\sum_{ P, Q \;:\; P \supsetneq Q}   \langle T (\Delta _{P}  f), \Delta _{Q} g\rangle \,, 
	\\ \label{e.=}
	 B _{\textup{=}} (f,g) &:= \sum_{P, Q \;:\; P =   Q}   \langle T \Delta _{P} f, \Delta _{Q} g\rangle \,,
	 \end{align} 
where the summations are restricted to cubes $P,Q$ that are contained in $Q_0$.
We will suppress this notationwise also in the sequel, but this fact will be used nevertheless.
 It is noteworthy that the martingale transform inequality is decisive in estimating 
 both of these terms. 

\section{The Term $ B _{\textup{above}}$} 

We address a book keeping issue.  For cube $ P$ set 
\begin{equation*}
	\widetilde \Delta _P f 
	:= \sum_{P'\in \textup{ch} (P) \setminus \mathcal S_{1}} 
	\frac {\langle f \rangle _{P'}} {\langle b ^{1} _{\pi _{\mathcal S_{1}}P' } \rangle_{P'}} \mathbf 1_{P'}
	-\frac {\langle f \rangle _{P}} {\langle b ^{1} _{\pi _{\mathcal S_{1}}P } \rangle_P} \mathbf 1_{P}\,. 
\end{equation*}
This is closely related to a half-twisted martingale difference associated with $ P$.  It suffices to show that for any $ S\in \mathcal S_{1}$
\begin{equation*}
\Bigl\lvert 
\mathbf{1}_{ \{S\not=Q_0\}}\cdot
\langle f \rangle _{S} \sum_{Q \;:\; Q\subset S} \langle T b ^{1} _S, \Delta _Qg\rangle 
+ 
\sum_{P \;:\; \pi _{\mathcal S_{1}}P=S} \sum_{Q\subsetneq P} \langle T  (b ^{1}_S \widetilde \Delta _P f), \Delta _Qg\rangle 
\Bigr\rvert \lesssim   \mathbf T _{\textup{loc}}    \lvert  S\rvert \,.  
\end{equation*}
Indeed, the sum of the left-hand side over $ S\in \mathcal S_1$ equals $ B _{\textup{above}} (f,g)$, and the collection $ \mathcal S_{1} $ 
is a Carleson sequence of cubes.  
The terms involving $ f$ only depend upon $ b ^{1}_S$, which is a convenience. 

The first term is easy to estimate.  The twisted martingale differences on $ g$ telescope, so that 
\begin{align*}
	\sum_{Q \;:\; Q\subset S}   \Delta _Qg & = g \mathbf 1_{S} - 
		\frac {\langle g \rangle _{S}} {\langle b ^{2} _{\pi _{\mathcal S_{2}}S } \rangle_S}  
		b ^{2} _{\pi _{\mathcal S_{2}}S }  \mathbf 1_{S} \,. 
\end{align*}
The ratio of averages is controlled, by construction. By the local $ Tb$ assumptions
and construction,
\begin{align*}
	\lvert \langle T b ^{1}_S, g \mathbf 1_{S} \rangle\rvert +
	\lvert \langle T b ^{1}_S, b ^{2} _{\pi _{\mathcal S ^{2}}S }  \mathbf 1_{S} \rangle\rvert \lesssim \mathbf T _{\textup{loc}} \lvert  S\rvert\,.  
\end{align*}

For the second term, the twisted martingale transform is the decisive point.  For pairs of cubes $ Q \subsetneq P$, 
let $ P_Q$ denote the child of $ P$ that contains $ Q$. The property of $ T$ being  perfect, and 
$ \Delta _Qg$ having integral zero, allows us to write 
\begin{align*}
 \langle T  (b ^{1}_S \widetilde \Delta _P f), \Delta _Qg\rangle & = 
 \langle T  (b ^{1}_S \widetilde \Delta _P f  \cdot \mathbf 1_{P_Q}), \Delta _Qg\rangle 
 \\
 &= \langle  \widetilde \Delta _P f   \rangle _{P_Q}
  \langle T  (b ^{1}_S  \mathbf 1_{P_Q}), \Delta _Qg\rangle 
  \\
  &= \langle  \widetilde \Delta _P f   \rangle _{P_Q}
  \langle T  b ^{1}_S, \Delta _Qg\rangle \,. 
\end{align*}
We have first restricted the argument of $ T$ to the cube $ P_Q$, pulled out the constant value of $\widetilde \Delta _P f$ 
on that cube, and finally extended the argument of $ T$ to the entire cube $ S$. 

Now, fix $ Q\subsetneq S$, and define a constant $ \varepsilon _Q$ by the formula 
\begin{equation*}
	\varepsilon _{Q}  := 
	\sum_{\substack{{ P  \;:\; P\supsetneq Q}\\ \pi_{\mathcal{S}_1} P=S}} \langle  \widetilde \Delta _P f   \rangle _{P_Q} \,. 
\end{equation*}
These numbers are bounded by a constant, since the sum  is  telescoping, and equals the difference of two 
$b$-averages of  $ f$ (or a single average, in case of $\pi_{\mathcal{S}_1}Q\subsetneq S$), which are bounded. 
We can  make a direct  appeal to the local $ Tb$ hypothesis. 
\begin{align*}
\Bigl\lvert 
\sum_{P \;:\; \pi _{\mathcal S_{1}}P=S} \sum_{Q\subsetneq P}
\langle  \widetilde \Delta _P f   \rangle _{P_Q} \langle T  b ^{1}_S, \Delta _Qg\rangle 
\Bigr\rvert 
&= 
\Bigl\lvert 
\sum_{Q\subsetneq S} \langle T  b ^{1}_S , \varepsilon _Q \Delta _Qg\rangle 
\Bigr\rvert
\\
& \le \mathbf T _{\textup{loc}} \lvert  S\rvert ^{1/p_2'} 
\Bigl\lVert 
\sum_{Q\subsetneq S} \varepsilon _Q \Delta _Qg 
\Bigr\rVert_{p_2}
\lesssim \mathbf T _{\textup{loc}} \lvert  S\rvert \,. 
\end{align*}
The twisted martingale transform inequality and the construction provide the last inequality. 

Indeed, let $ S_2$ be the $ \mathcal S_2$ parent of $ S$, and set $ \mathcal R _1 := \{S_2\}$. 
Let $ \mathcal R _2$ be the $ \mathcal S_2$ children of $S_2$
strictly contained in $ S$, and inductively set $ \mathcal R _{k+1}$ to the $ \mathcal S_2$ children of cubes $ R \in \mathcal R _{k}$. 
Each function below is a twisted martingale transform of $ g$
\begin{equation*}
	\gamma _{R} := \sum_{\substack{Q \;:\; \pi _{\mathcal S_2} Q=R\\ Q\subsetneq S }} \varepsilon _Q \Delta _{Q} g \,,   \qquad 
	R \in  \bigcup _{k=1} ^{\infty } \mathcal R _{k} \,. 
\end{equation*}
There holds $ \lVert \gamma _R\rVert_{p_2} \lesssim \lvert  R \cap S\rvert ^{1/p_2} $, by the martingale transform inequality 
and the fact that $ g$ is a bounded function.  Moreover,  from \eqref{e.generations_no_f}, it follows that  
\begin{equation*}
	\sum_{R\in \mathcal R _{k}} \lvert   R  \cap S\rvert \lesssim\tau^k \lvert  S\rvert \,, 
\end{equation*}
where $ 0< \tau < 1$ is fixed.  Hence, 
\begin{align}  \label{e.888}
\Bigl\lVert 
\sum_{Q\subsetneq S} \varepsilon _Q \Delta _Qg 
\Bigr\rVert_{p_2} ^{p_2} &= 
\Bigl\lVert  \sum_{k=1} ^{\infty } 
\sum_{ R \in \mathcal R _{k}} k ^{-1+1} \gamma _{R} 
\Bigr\rVert_{p_2}^{p_2}
\\
& \lesssim \sum_{k=1} ^{\infty } 
k ^{p_2} \sum_{R\in \mathcal R _{k}} \lVert \gamma _{R}\rVert_{p_2} ^{p_2} 
\\ & \lesssim \lvert  S\rvert \sum_{k=1} ^{\infty } k ^{p_2} \tau^k  
\lesssim \lvert  S\rvert \,.  
\end{align}
This completes the analysis of the above form.


\section{The Diagonal Term} 

 One can can compare this argument to that of \cite{1011.1747}*{Section 8.1}.  
Before beginning the main thrust of the argument, 
a particular consequence of the martingale transform inequality is needed.  Using the notation of Theorem~\ref{t.LM}, set 
for $ j=1,2$,
\begin{align}\label{e.Box}
	\Box ^{j} _Q h := \lvert   D ^{j} _Q h \rvert +
	\begin{cases}
	\mathbf 1_{Q}  & \textup{a child of $ Q$ is in $ \mathcal S _{j}$} 
	\\
	0 & \textup{otherwise} 
 	\end{cases}
	\end{align}
where $ D ^{j} _{Q}$ is defined as in \eqref{e.D}, with 
$\mathcal{S}_0^j :=\pi_{\mathcal{S}_j}S$, terminal  cubes 
$\mathcal{T}^j:=\textup{ch}_{\mathcal{S}_j}(S_0^j)$, and
function $b^{j}:=b^{j} _{S_0^j}$.  Note that second summand accounts for the missing terminal cubes in 
definition of $ D ^{j}_Q$.  

By a randomization argument, the half-twisted inequality of Theorem~\ref{t.LM}, and the Carleson measure property of the cubes,  there holds 
\begin{equation} \label{e.Box<}
\Bigl\lVert 
\Bigl[
\sum_{\substack{Q \;:\;  Q\subset Q _0 }}   
(\Box ^{1}  _{Q} f ) ^2 
\Bigr] ^{1/2} 
\Bigr\rVert_{q}
\lesssim \lvert  Q_0\rvert ^{1/q}\,, \qquad 1< q < \infty  \,. 
\end{equation} 
The same inequality holds for $ g$.  

To control the diagonal term, it therefore suffices to show that 
\begin{equation*}
	\lvert  \langle T \Delta _Q f ,  \Delta _Q g\rangle \rvert 
	\lesssim (1 + \mathbf T _{\textup{loc}}) 
	\sum_{Q^1,Q^2\in\textup{ch}(Q)}
	\langle  \Box ^{1}_Q f\rangle_{Q^1} |Q| \langle \Box ^{2}_Q g \rangle_{Q^2} \,, \qquad Q\subset Q_0 \,. 
\end{equation*}
For cube $ Q$, and child $ Q^1$ of $ Q$, $  \Delta _Q f \mathbf 1_{Q^1}$ is either 
a multiple of $ b ^{1} _{\pi _{\mathcal S _{1}} Q} \mathbf 1_{Q^1}$, or, in the exclusive case that $ Q^1$ is also a stopping cube,  a linear combination of this function and $ b ^{1} _{Q^1} $.  In both cases, the coefficents in the linear combination are dominated by 
a constant times  $ \langle \Box ^{1}_Q f\rangle _{Q^1}$.   
Therefore, the control of the term above follows from this Lemma. 

\begin{lemma}\label{l.matrix_est}
Suppose that $Q\subset Q_0$. 
Then, if $Q^j$ is a children of $Q$
and $b^j\in\{b^j_{\pi_{\mathcal{S}_j}Q}, b_{Q^j}^j\}$ 
with $j=1,2$,
\[
	\big\lvert \langle T(b^1 \mathbf 1_{Q^1}),b^2\mathbf 1_{Q^2}\rangle \big\rvert \lesssim(1+\mathbf{T}_{\textup{loc}})\lvert Q\rvert.
\]
\end{lemma}

\begin{proof}
Let us consider the case  $b^j=b^j_{\pi_{\mathcal{S}_j}Q}$, $j=1,2$. The other
cases are similar but easier.
Suppose first that $Q^1\not=Q^2$. Then, since $T$ is perfect, we see that $K$ is constant on $Q^2\times Q^1$.
Hence, by denoting the midpoint of $Q^j$ by $x_{Q^j}$,
\begin{align*}
\big\lvert \langle T(b^1 \mathbf 1_{Q^1}),b^2\mathbf 1_{Q^2}\rangle\big\rvert
&=\lvert K(x_{Q^2},x_{Q^1})\rvert \cdot \bigg\lvert \int_{Q^1} b^1_{\pi_{\mathcal{S}_1}Q}(y)\,dy\bigg\rvert\cdot \bigg\lvert
\int_{Q^2} b^2_{\pi_{\mathcal{S}_2}Q}(x)\,dx\bigg\rvert\\
&\lesssim |Q|.
\end{align*}
In the last step we used the kernel size estimate.

Then we suppose that $Q^1=Q^2$.
We let $b_{Q^1}^2$ be the $p_2$-accretive
function, associated with the cube $Q^1$.
It suffices to estimate the following terms,
\begin{equation}\label{e.t_split}
\big\lvert \langle T(b^1 \mathbf 1_{Q^1}),\langle b^2\rangle_{Q^1} b_{Q^1}^2\rangle\big\rvert
+\big\lvert \langle T(b^1 \mathbf 1_{Q^1}),b^2\mathbf 1_{Q^1}-\langle b^2\rangle_{Q^1} b_{Q^1}^2\rangle\big\rvert\,.
\end{equation}
The first term is bounded by
\begin{align*}
\lvert \langle b^2\rangle_{Q^1}\rvert\cdot
\big\lvert \langle b^1 \mathbf 1_{Q^1},T^*(b_{Q^1}^2)\rangle\big\rvert
\lesssim \mathbf{T}_{\textup{loc}}\lvert Q\rvert\,.
\end{align*}
Here we used H\"older's inequality, with both exponents $p_2$ and $p_1$. 
To estimate the second term in \eqref{e.t_split}, the crucial step
is to remove the characteristic function
$\mathbf{1}_{Q^1}$ from within $T(b^1\mathbf{1}_{Q^1})$.
For this purpose, let us observe that the function $B^2:=b^2\mathbf 1_{Q^1}-\langle b^2\rangle_{Q^1} b_{Q^1}^2$
is supported on $Q^1$ and it has zero integral. By assumption
that $T$ is perfect,
\begin{align*}
\big\lvert \langle T(b^1 \mathbf 1_{Q^1}),B^2\rangle \big\rvert
=\big\lvert \langle T(b^1),B^2\rangle \big\rvert\lesssim \mathbf{T}_{\textup{loc}}|Q|.
\end{align*}
In the last step, we split the dual form
in two other forms and use H\"older's inequality, with exponent $p_2$, for
the individual forms separately.
\end{proof}

\end{document}